\documentclass[reqno]{amsproc}%
\usepackage{amsmath}
\usepackage{amssymb}
\usepackage{amsfonts}
\usepackage{graphicx}%
\providecommand{\U}[1]{\protect\rule{.1in}{.1in}}
\theoremstyle{plain}

\newtheorem{definition}{Definition}

\newtheorem{lemma}{Lemma}

\newtheorem{remark}{Remark}

\newtheorem{theorem}{Theorem}
\numberwithin{equation}{section}
\begin{document}
\title{A Dunkl-Gamma Type Operator in Terms of Generalization of Two-Variable Hermite Polynomials}
\author{Bayram \c{C}ekim}
\curraddr{Gazi University, Faculty of Science, Department of Mathematics, 06500,
Teknikokullar, Ankara, Turkey}
\email{bayramcekim@gazi.edu.tr}
\author{Rabia Akta\c{s}}
\address{Ankara University, Faculty of Science, Department of Mathematics, 06100,
Tando\u{g}an, Ankara, Turkey}
\email{raktas@science.ankara.edu.tr}
\author{Fatma Taşdelen}
\address{Ankara University, Faculty of Science, Department of Mathematics, 06100,
Tando\u{g}an, Ankara, Turkey}
\email{tasdelen@science.ankara.edu.tr}

\begin{abstract}
The goal of this paper is to present a Dunkl-Gamma type operator with the help
of generalization of the two-variable Hermite polynomials and to derive its
approximating properties via the classical modulus of continuity, second
modulus of continuity and Peetre's $K$-functional.

\end{abstract}
\maketitle

\section{Introduction}

By now, several research workers have investigated linear positive operators
and their approximation properties, see for instance \cite{Ata-Buyuk},
\cite{Atakut}, \cite{Ciupa}, \cite{Gupta}, \cite{Szasz}, \cite{Stancu} and
references so on.\textbf{\ }Furthermore, many authors have studied linear
positive operators containing generating functions and given some
approximation properties of these operators. To see such operators, we give
the references such\ as Alt\i n et. al \cite{ADT}, Do\u{g}ru et. al
\cite{DOT}, Krech \cite{K}, Olgun et. al \cite{OIT}, Sucu et. al \cite{Sucu et
al.}, Ta\c{s}delen et. al \cite{TAA}, Varma et. al \cite{Varma et al., VT}.

Latterly, with the help of Dunkl exponential function, several authors have
defined some linear positive operators. First of them is a Dunkl analogue of
Sz\'{a}sz operators given\ in \cite{Sucu} as follows:%
\begin{equation}
S_{n}^{\ast}\left(  g;x\right)  =\frac{1}{e_{\nu}\left(  nx\right)  }%
\sum_{k=0}^{\infty}\frac{\left(  nx\right)  ^{k}}{\gamma_{\nu}\left(
k\right)  }g\left(  \frac{k+2\nu\theta_{k}}{n}\right)  \ ;\ n\in%
\mathbb{N}
,~\nu,\ x\in\lbrack0,\infty)\ \label{1}%
\end{equation}
for $g\in C[0,\infty).$ Here the Dunkl exponential function is defined by%
\begin{equation}
e_{\nu}\left(  x\right)  =\sum_{k=0}^{\infty}\frac{x^{k}}{\gamma_{\nu}\left(
k\right)  } \label{7}%
\end{equation}
for $\nu>-\frac{1}{2}$ and the coefficients $\gamma_{\nu}$ are given by%
\begin{equation}
\gamma_{\nu}\left(  2k\right)  =\frac{2^{2k}k!\Gamma\left(  k+\nu+1/2\right)
}{\Gamma\left(  \nu+1/2\right)  }\text{ and }\gamma_{\nu}\left(  2k+1\right)
=\frac{2^{2k+1}k!\Gamma\left(  k+\nu+3/2\right)  }{\Gamma\left(
\nu+1/2\right)  } \label{3}%
\end{equation}
Also, for the coefficients $\gamma_{\nu},$ the following recursion relation
holds%
\begin{equation}
\frac{\gamma_{\nu}\left(  k+1\right)  }{\gamma_{\nu}\left(  k\right)
}=\left(  2\nu\theta_{k+1}+k+1\right)  ,\text{ }k\in%
\mathbb{N}
_{0}, \label{4}%
\end{equation}
where $\theta_{k}$ is defined by%
\begin{equation}
\theta_{k}=\left\{
\begin{array}
[c]{cc}%
0, & if\text{ }k=2p\\
1, & if\text{ }k=2p+1
\end{array}
\right.  \label{5}%
\end{equation}
for $p\in%
\mathbb{N}
_{0}$ in \cite{Rosenblum}. Then, \.{I}\c{c}\"{o}z and \c{C}ekim have given a
Stancu-type generalization of Sz\'{a}sz-Kantorovich operators and
$q-$Sz\'{a}sz operators with the help of the Dunkl exponential function in
\cite{IC, icoz}.

Next, Wafi and Rao \cite{wafi} has introduced Sz\'{a}sz--Gamma operators based
on Dunkl analogue as%
\begin{equation}
D_{n}^{f}(x)=\frac{1}{e_{\mu}(nx)}\sum\limits_{k=0}^{\infty}\frac{(nx)^{k}%
}{\gamma_{\mu}(k)}\frac{n^{k+2\mu\theta_{k}+\lambda+1}}{\Gamma(k+2\mu
\theta_{k}+\lambda+1)}%
{\textstyle\int_{0}^{\infty}}
t^{k+2\mu\theta_{k}+\lambda}e^{-nt}f(t)dt, \label{10001}%
\end{equation}
where $\lambda\geq0$ and $\Gamma$ is Gamma function defined by%
\begin{equation}
\Gamma(x)=%
{\textstyle\int_{0}^{\infty}}
t^{x-1}e^{-t}dt\text{ for }x>0. \label{10003}%
\end{equation}

Finally, Akta\c{s} et. al \cite{BRF} has introduced the operator $T_{n}(f;x)$
for $n\in%
\mathbb{N}
$%
\[
T_{n}(f;x):=\frac{1}{e^{\alpha x^{2}}e_{\mu}(nx)}\sum\limits_{k=0}^{\infty
}\frac{h_{k}^{\mu}(n,\alpha)}{\gamma_{\mu}(k)}x^{k}f\left(  \frac{k+2\mu
\theta_{k}}{n}\right)
\]
where $\alpha\geq0,\mu\geq0$ and $x\in\left[  0,\infty\right)  ,$ via the
Dunkl generalization of two-variable Hermite polynomials, $h_{n}^{\mu}%
(\xi,\alpha)$ in \cite{BC} defined as follows%
\begin{equation}
\sum\limits_{n=0}^{\infty}\frac{h_{n}^{\mu}(\xi,\alpha)}{\gamma_{\mu}(n)}%
t^{n}=e^{\alpha t^{2}}e_{\mu}(\xi t). \label{a1}%
\end{equation}

Here%
\[
h_{n}^{\mu}(\xi,\alpha)=\dfrac{\gamma_{\mu}(n)H_{n}^{\mu}(\xi,\alpha)}{n!}%
\]
and $H_{n}^{\mu}(\xi,\alpha)$ has the following explicit representation%
\[
H_{n}^{\mu}(\xi,\alpha)=n!\sum\limits_{k=0}^{\left[  \frac{n}{2}\right]
}\frac{\alpha^{k}\xi^{n-2k}}{k!\gamma_{\mu}(n-2k)}.
\]
We note that $H_{n}^{\mu}(\xi,\alpha)$ reduces to the two-variable Hermite
polynomials defined by%
\begin{equation}
H_{n}(\xi,\alpha)=n!\sum\limits_{k=0}^{\left[  \frac{n}{2}\right]  }%
\frac{\alpha^{k}\xi^{n-2k}}{k!(n-2k)!} \label{xx}%
\end{equation}
as $\mu=0,\ $see detail \cite{Appell}. In the case of $\mu=0,$ the operator
$T_{n}(f;x)$ gives the operator $G_{n}^{\alpha}\left(  f;x\right)  $ defined
by Krech \cite{K} as follows%
\[
G_{n}^{\alpha}\left(  f;x\right)  =e^{-\left(  nx+\alpha x^{2}\right)  }%
\sum\limits_{k=0}^{\infty}\frac{x^{k}}{k!}H_{k}(n,\alpha)f\left(  \frac{k}%
{n}\right)  ,~~~x\in\mathbb{R}_{0}^{+}:=\left[  0,\infty\right)  ,
\]
where $H_{k}$ is the two variable Hermite polynomial in (\ref{xx}).
Furthermore, recently, some sequences of Dunkl operators and Dunkl-Gamma type
operators in terms of Appell polynomials have been defined and approximation
properties of these operators have been investigated \cite{Sucu2, TSA}.

The paper is organized as follows. In the next section, we introduce a
Dunkl-Gamma type operator consisting of the generalization of two-variable
Hermite polynomials. In the third section, the rates of convergence of the
operator are obtained by means of the classical modulus of continuity, second
modulus of continuity, Peetre's $K$-functional and the Lipschitz class
$Lip_{M}\left(  \gamma\right)  .$

\section{The Dunkl-Gamma Type Operator}

Firstly, before we introduce our operator, let us give some features and
results related to $h_{n}^{\mu}(\xi,\alpha)$ generated by the Dunkl
generalization of two-variable Hermite polynomials in (\ref{a1}).

We first recall the following definition and lemma in \cite{Rosenblum}.

\begin{definition}
\cite{Rosenblum} Assume that $%
\mu
\in\mathbb{C},~x\in\mathbb{C}$. On all entire functions $\varphi$ on
$\mathbb{C}$, Rosenblum defines the linear operator $\mathbb{D}_{\mu}$ as
follows:%
\begin{equation}
\mathbb{D}_{\mu,x}(\varphi(x))=\left(  \mathbb{D}_{\mu}\varphi\right)
(x)=\varphi^{^{\prime}}(x)+\frac{\mu}{x}(\varphi(x)-\varphi(-x)),\ x\in
\mathbb{C}. \label{6}%
\end{equation}

\end{definition}

\begin{lemma}
\label{Lemma1} \cite{Rosenblum} Assume that $\varphi,\psi$ are entire
functions. With the help of the linear operator $\mathbb{D}_{\mu}$, the
following relations are satisfied:%
\[%
\begin{array}
[c]{cl}%
i) & \mathbb{D}_{\mu}^{j}:x^{n}\rightarrow\frac{\gamma_{\mu}(n)}{\gamma_{\mu
}(n-j)}x^{n-j},\ j=0,1,2,...,n\ (n\in\mathbb{N)};~\mathbb{D}_{\mu}%
^{j}:1\rightarrow0,\\
& \\
ii) & \mathbb{D}_{\mu}(\varphi\psi)=\mathbb{D}_{\mu}(\varphi)\psi
+\varphi\mathbb{D}_{\mu}(\psi),\ \text{if }\varphi\text{ is an even
function,}\\
& \\
iii) & \mathbb{D}_{\mu}:e_{\mu}(\lambda x)\rightarrow\lambda e_{\mu}(\lambda
x).
\end{array}
\]
By using this definition and Lemma \ref{Lemma1}, the results in the next lemma
hold true (see detail \cite{BRF}).
\end{lemma}

\begin{lemma}
\label{Lemma2} \cite{BRF} $h_{n}^{\mu}(\xi,\alpha)$ has the following results
\[%
\begin{array}
[c]{cl}%
(i) & \sum\limits_{n=0}^{\infty}\frac{h_{n+1}^{\mu}(\xi,\alpha)}{\gamma_{\mu
}(n)}t^{n}=(\xi+2\alpha t)e^{\alpha t^{2}}e_{\mu}(\xi t),\\
(ii) & \sum\limits_{n=0}^{\infty}\frac{h_{n+2}^{\mu}(\xi,\alpha)}{\gamma_{\mu
}(n)}t^{n}=(\xi^{2}+4\xi\alpha t+4\alpha^{2}t^{2}+2\alpha)e^{\alpha t^{2}%
}e_{\mu}(\xi t)+4\alpha\mu e^{\alpha t^{2}}e_{\mu}(-\xi t).
\end{array}
\]

\end{lemma}

Now we can define our operator as follows:

\begin{definition}
Via $h_{n}^{\mu}(\xi,\alpha)$ given in (\ref{a1}), we consider the operator
$\mathcal{S}_{n}(f;x),$ $n\in%
\mathbb{N}
$ given by
\begin{equation}
\mathcal{S}_{n}(f;x):=\frac{1}{e^{\alpha x^{2}}e_{\mu}(nx)}\sum\limits_{k=0}%
^{\infty}\frac{h_{k}^{\mu}(n,\alpha)}{\gamma_{\mu}(k)}x^{k}\frac
{n^{k+2\mu\theta_{k}+\lambda+1}}{\Gamma(k+2\mu\theta_{k}+\lambda+1)}%
{\textstyle\int_{0}^{\infty}}
t^{k+2\mu\theta_{k}+\lambda}e^{-nt}f(t)dt \label{1001}%
\end{equation}
where$\ \alpha\geq0,\ \mu>\frac{-1}{2},\ \lambda\geq0$ and $x\in\left[
0,\infty\right)  .$ We note that the operator in (\ref{1001}) is positive and
linear. For $\alpha=0,$ it reduces to $D_{n}^{f}(x)$ given by (\ref{10001}).
\end{definition}

\begin{lemma}
\label{Lemma 3} The following equations can be derived from the definition of
the operator $\mathcal{S}_{n}(f;x)$:%
\[%
\begin{array}
[c]{cl}%
(i) & \mathcal{S}_{n}(1;x)=1,\\
(ii) & \mathcal{S}_{n}(t;x)=x+\frac{2\alpha x^{2}}{n}+\frac{\lambda+1}{n},\\
(iii) &
\begin{array}
[c]{l}%
\mathcal{S}_{n}(t^{2};x)=\frac{x^{2}}{n^{2}}\left\{  n^{2}+4n\alpha
x+4\alpha^{2}x^{2}+2\alpha+4\alpha\mu\frac{e_{\mu}(-nx)}{e_{\mu}(nx)}\right\}
\\
+\frac{2\mu x}{n^{2}}(n-2\alpha x)\frac{e_{\mu}(-nx)}{e_{\mu}(nx)}%
+\frac{2(\lambda+2)}{n^{2}}(n+2\alpha x)x+\frac{(\lambda+1)(\lambda+2)}{n^{2}%
}.
\end{array}
\end{array}
\]

\end{lemma}

\begin{proof}
From the definition of Gamma function in (\ref{10003}), we have%
\[%
{\textstyle\int_{0}^{\infty}}
t^{k+2\mu\theta_{k}+\lambda}e^{-nt}dt=\frac{\Gamma(k+2\mu\theta_{k}%
+\lambda+1)}{n^{k+2\mu\theta_{k}+\lambda+1}}.
\]
By using the above equation and the generating function in (\ref{a1}), we get
the relation $(i)$. Using the definition of Gamma function again, we have%
\[%
{\textstyle\int_{0}^{\infty}}
t^{k+2\mu\theta_{k}+\lambda+1}e^{-nt}dt=\frac{\Gamma(k+2\mu\theta_{k}%
+\lambda+2)}{n^{k+2\mu\theta_{k}+\lambda+2}}.
\]
Thus we get the relations%
\begin{align*}
\mathcal{S}_{n}(t;x)  &  =\frac{1}{e^{\alpha x^{2}}e_{\mu}(nx)}\sum
\limits_{k=0}^{\infty}\frac{h_{k}^{\mu}(n,\alpha)}{\gamma_{\mu}(k)}x^{k}%
\frac{n^{k+2\mu\theta_{k}+\lambda+1}}{\Gamma(k+2\mu\theta_{k}+\lambda+1)}%
\frac{\Gamma(k+2\mu\theta_{k}+\lambda+2)}{n^{k+2\mu\theta_{k}+\lambda+2}}\\
&  =\frac{1}{ne^{\alpha x^{2}}e_{\mu}(nx)}\sum\limits_{k=0}^{\infty}%
(k+2\mu\theta_{k}+\lambda+1)\frac{h_{k}^{\mu}(n,\alpha)}{\gamma_{\mu}(k)}%
x^{k}\\
&  =\frac{1}{ne^{\alpha x^{2}}e_{\mu}(nx)}\sum\limits_{k=0}^{\infty}%
(k+2\mu\theta_{k})\frac{h_{k}^{\mu}(n,\alpha)}{\gamma_{\mu}(k)}x^{k}%
+\frac{(\lambda+1)}{ne^{\alpha x^{2}}e_{\mu}(nx)}\sum\limits_{k=0}^{\infty
}\frac{h_{k}^{\mu}(n,\alpha)}{\gamma_{\mu}(k)}x^{k}.
\end{align*}
The second series in right hand side of the above equation from the generating
function in (\ref{a1}) is $\frac{(\lambda+1)}{n}.$ Also, if we use the
recursion relation in (\ref{4}) for the first term, we get%
\[
\mathcal{S}_{n}(t;x)=\frac{1}{ne^{\alpha x^{2}}e_{\mu}(nx)}\sum\limits_{k=1}%
^{\infty}\frac{h_{k}^{\mu}(n,\alpha)}{\gamma_{\mu}(k-1)}x^{k}+\frac
{(\lambda+1)}{n}.
\]
While we are substituting $k$ by $k+1$ and using Lemma \ref{Lemma2} $(i)$, we
arrive at the relation $(ii)$. From the definition of Gamma function in
(\ref{10003}) again, the following equality holds%
\[%
{\textstyle\int_{0}^{\infty}}
t^{k+2\mu\theta_{k}+\lambda+2}e^{-nt}dt=\frac{\Gamma(k+2\mu\theta_{k}%
+\lambda+3)}{n^{k+2\mu\theta_{k}+\lambda+3}},
\]
from which, it follows%
\begin{align*}
\mathcal{S}_{n}(t^{2};x)  &  =\frac{1}{e^{\alpha x^{2}}e_{\mu}(nx)}%
\sum\limits_{k=0}^{\infty}\frac{h_{k}^{\mu}(n,\alpha)}{\gamma_{\mu}(k)}%
x^{k}\frac{n^{k+2\mu\theta_{k}+\lambda+1}}{\Gamma(k+2\mu\theta_{k}+\lambda
+1)}\frac{\Gamma(k+2\mu\theta_{k}+\lambda+3)}{n^{k+2\mu\theta_{k}+\lambda+3}%
}\\
&  =\frac{1}{n^{2}e^{\alpha x^{2}}e_{\mu}(nx)}\sum\limits_{k=0}^{\infty
}(k+2\mu\theta_{k}+\lambda+1)(k+2\mu\theta_{k})\frac{h_{k}^{\mu}(n,\alpha
)}{\gamma_{\mu}(k)}x^{k}\\
&  +\frac{(\lambda+2)}{n^{2}e^{\alpha x^{2}}e_{\mu}(nx)}\sum\limits_{k=0}%
^{\infty}(k+2\mu\theta_{k})\frac{h_{k}^{\mu}(n,\alpha)}{\gamma_{\mu}(k)}%
x^{k}\\
&  +\frac{(\lambda+2)(\lambda+1)}{n^{2}e^{\alpha x^{2}}e_{\mu}(nx)}%
\sum\limits_{k=0}^{\infty}\frac{h_{k}^{\mu}(n,\alpha)}{\gamma_{\mu}(k)}x^{k}.
\end{align*}
The third term in right hand side of the above equation from the generating
function in (\ref{a1}) is $\frac{(\lambda+2)(\lambda+1)}{n^{2}}.$ Also by
taking into account the recursion relation in (\ref{4}) the for first and
second series, we obtain%
\begin{align*}
\mathcal{S}_{n}(t^{2};x)  &  =\frac{x}{n^{2}e^{\alpha x^{2}}e_{\mu}(nx)}%
\sum\limits_{k=0}^{\infty}(k+2\mu\theta_{k+1}+\lambda+2)\frac{h_{k+1}^{\mu
}(n,\alpha)}{\gamma_{\mu}(k)}x^{k}\\
&  +\frac{(\lambda+2)x}{n^{2}e^{\alpha x^{2}}e_{\mu}(nx)}\sum\limits_{k=0}%
^{\infty}\frac{h_{k+1}^{\mu}(n,\alpha)}{\gamma_{\mu}(k)}x^{k}+\frac
{(\lambda+2)(\lambda+1)}{n^{2}}.
\end{align*}

Using the equation%
\begin{equation}
\theta_{k+1}=\theta_{k}+(-1)^{k} \label{x}%
\end{equation}
in \cite{Rosenblum}, it yields%
\begin{align*}
\mathcal{S}_{n}(t^{2};x)  &  =\frac{x}{n^{2}e^{\alpha x^{2}}e_{\mu}(nx)}%
\sum\limits_{k=0}^{\infty}(k+2\mu\theta_{k})\frac{h_{k+1}^{\mu}(n,\alpha
)}{\gamma_{\mu}(k)}x^{k}\\
&  +\frac{2\mu x}{n^{2}e^{\alpha x^{2}}e_{\mu}(nx)}\sum\limits_{k=0}^{\infty
}\frac{h_{k+1}^{\mu}(n,\alpha)}{\gamma_{\mu}(k)}(-x)^{k}\\
&  +\frac{2(\lambda+2)x}{n^{2}e^{\alpha x^{2}}e_{\mu}(nx)}\sum\limits_{k=0}%
^{\infty}\frac{h_{k+1}^{\mu}(n,\alpha)}{\gamma_{\mu}(k)}x^{k}+\frac
{(\lambda+2)(\lambda+1)}{n^{2}}.
\end{align*}
Finally using the recursion relation in (\ref{4}) in the first series, from
Lemma \ref{Lemma2} $(i)$ for the second and third series and Lemma
\ref{Lemma2} $(ii)$ for the first series, we complete the proof of $(iii)$.
\end{proof}

\begin{remark}
In case of $\alpha=0,$ the results of Lemma \ref{Lemma 3} reduce to the
results in the paper of Wafi and Rao in \cite{wafi}.
\end{remark}

\begin{lemma}
\label{Lemma 4}From the results of Lemma \ref{Lemma 3} and the linearity of
the operator, we can obtain the next results for $\mathcal{S}_{n}$ operator%
\begin{align}
\Lambda_{1}  &  =\mathcal{S}_{n}(t-x;x)=\frac{2\alpha x^{2}+\lambda+1}%
{n},\nonumber\\
\Lambda_{2}  &  =\mathcal{S}_{n}(\left(  t-x\right)  ^{2};x)\nonumber\\
&  =\frac{1}{n}\left[  \frac{x^{2}}{n}\left(  4x^{2}\alpha^{2}+4\lambda
\alpha+10\alpha\right)  +2x\left(  \mu\tfrac{e_{\mu}(-nx)}{e_{\mu}%
(nx)}+1\right)  +\tfrac{(\lambda+1)(\lambda+2)}{n}\right]  . \label{A}%
\end{align}
Taking into account the inequality $\left\vert \tfrac{e_{\mu}(-x)}{e_{\mu}%
(x)}\right\vert \leq1$ for $x\geq0$ and $\mu>\frac{-1}{2}$ and $\tfrac{e_{\mu
}(-x)}{e_{\mu}(x)}\rightarrow0$ as $x\rightarrow\infty$ in \cite{Milo}, we
have the following theorem.
\end{lemma}

\begin{theorem}
\label{Theorem 4} Assume that the function $g$ on the interval $[0,\infty)$ is
uniformly continuous bounded function. For each function $g$ on $[0,\infty),$
we can give%
\[
\mathcal{S}_{n}\left(  g;x\right)  \overset{\text{uniformly}}%
{\rightrightarrows}g\left(  x\right)
\]
on each compact set $A\subset$ $[0,\infty)$\ when $n\rightarrow\infty$.
\end{theorem}

\begin{proof}
In view of Lemma 3%
\[
\underset{n\rightarrow\infty}{\lim}\mathcal{S}_{n}\left(  t^{i};x\right)
=x^{i}~~,~~i=0,1,2
\]
is verified where the convergence holds uniformly in each compact subset of
$[0,\infty)$. Then, using well known Korovkin Theorem in \cite{Korovkin}, we
give the desired result.
\end{proof}

\section{The Convergence Rates of Operator $\mathcal{S}_{n}$}

In this part, we obtain some rates of convergence of the operator
$\mathcal{S}_{n}$.

\begin{theorem}
\label{Theorem 6}If $h\in Lip_{M}\left(  \gamma\right)  $, which satisfies the
inequality%
\[
\left\vert h\left(  s\right)  -h\left(  t\right)  \right\vert \leq M\left\vert
s-t\right\vert ^{\gamma}%
\]
where $s,t\in\lbrack0,\infty),\ 0<\gamma\leq1$ and $M>0,$ we have%
\[
\left\vert \mathcal{S}_{n}\left(  h;x\right)  -h\left(  x\right)  \right\vert
\leq M\left(  \Lambda_{2}\right)  ^{\gamma/2}%
\]
where $\Lambda_{2}$ is given in Lemma \ref{Lemma 4}.
\end{theorem}

\begin{proof}
From the linearity of operator and $h\in Lip_{M}\left(  \gamma\right)  ,$ we
get%
\[
\left\vert \mathcal{S}_{n}\left(  h;x\right)  -h\left(  x\right)  \right\vert
\leq\mathcal{S}_{n}\left(  \left\vert h\left(  t\right)  -h\left(  x\right)
\right\vert ;x\right)  \leq M\mathcal{S}_{n}\left(  \left\vert t-x\right\vert
^{\gamma};x\right)  .
\]
Under favour of H\"{o}lder's inequality and Lemma \ref{Lemma 4}, we can give
the following required inequality%
\[
\left\vert \mathcal{S}_{n}\left(  h;x\right)  -h\left(  x\right)  \right\vert
\leq M\left[  \Lambda_{2}\right]  ^{\frac{\gamma}{2}}.
\]

\end{proof}

\begin{theorem}
\label{Theorem 7}The operator $\mathcal{S}_{n}$ in (\ref{1001}) satisfies the
inequality%
\[
\left\vert \mathcal{S}_{n}\left(  g;x\right)  -g\left(  x\right)  \right\vert
\leq\left(  1+\sqrt{\tfrac{x^{2}}{n}\left(  4x^{2}\alpha^{2}+4\lambda
\alpha+10\alpha\right)  +2x\left(  \mu\tfrac{e_{\mu}(-nx)}{e_{\mu}%
(nx)}+1\right)  +\tfrac{(\lambda+1)(\lambda+2)}{n}}\right)  \omega\left(
g;\frac{1}{\sqrt{n}}\right)  ,
\]
where $g\in\widetilde{C}[0,\infty),$ which is the space of uniformly
continuous functions on $[0,\infty),\ $and the modulus of continuity is
defined by
\begin{equation}
\omega\left(  g;\delta\right)  :=\sup\limits_{\substack{s,t\in\lbrack0,\infty)
\\\left\vert s-t\right\vert \leq\delta}}\left\vert g\left(  s\right)
-g\left(  t\right)  \right\vert \label{12}%
\end{equation}
for $g\in\widetilde{C}[0,\infty).$
\end{theorem}

\begin{proof}
Firstly we note that the modulus of continuity verifies the following
inequality
\begin{equation}
\left\vert g\left(  t\right)  -g\left(  x\right)  \right\vert \leq
\omega\left(  g;\delta\right)  \left(  \frac{\left\vert t-x\right\vert
}{\delta}+1\right)  . \label{a8}%
\end{equation}
Under favour of the linearity of operator, Cauchy-Schwarz's inequality, and
Lemma \ref{Lemma 4}, respectively, it follows%
\begin{align*}
\left\vert \mathcal{S}_{n}\left(  g;x\right)  -g\left(  x\right)  \right\vert
&  \leq\mathcal{S}_{n}\left(  \left\vert g\left(  t\right)  -g\left(
x\right)  \right\vert ;x\right) \\
&  \leq\left(  1+\frac{1}{\delta}\mathcal{S}_{n}\left(  \left\vert
t-x\right\vert ;x\right)  \right)  \omega\left(  g;\delta\right) \\
&  \leq\left(  1+\frac{1}{\delta}\sqrt{\Lambda_{2}}\right)  \omega\left(
g;\delta\right)  .
\end{align*}
By choosing $\delta=\frac{1}{\sqrt{n}}$, we complete the proof.
\end{proof}

\begin{lemma}
\label{Lemma 8}For $h\in C_{B}^{2}[0,\infty)$, which is denoted by%
\begin{equation}
C_{B}^{2}[0,\infty)=\{h\in C_{B}[0,\infty):h^{\prime},h^{\prime\prime}\in
C_{B}[0,\infty)\}\label{16}%
\end{equation}
with the norm
\[
\left\Vert h\right\Vert _{C_{B}^{2}[0,\infty)}=\left\Vert h\right\Vert
_{C_{B}[0,\infty)}+\left\Vert h^{\prime}\right\Vert _{C_{B}[0,\infty
)}+\left\Vert h^{\prime\prime}\right\Vert _{C_{B}[0,\infty)}%
\]
where $C_{B}[0,\infty)$ is the space of continuous and bounded functions on
$[0,\infty)$ with the norm
\[
\left\Vert h\right\Vert _{C_{B}[0,\infty)}=\sup_{x\in\lbrack0,\infty
)}\left\vert h(x)\right\vert ~,
\]
the following inequality holds true%
\begin{equation}
\left\vert \mathcal{S}_{n}\left(  h;x\right)  -h\left(  x\right)  \right\vert
\leq\frac{\Lambda_{2}^{\frac{1}{2}}}{2}\left(  2+\Lambda_{2}^{\frac{1}{2}%
}\right)  \left\Vert h\right\Vert _{C_{B}^{2}[0,\infty)},\label{17}%
\end{equation}
where $\Lambda_{2}$ is given by in Lemma \ref{Lemma 4}.
\end{lemma}

\begin{proof}
With the help of the Taylor's series of the function $h$, it follows that%
\[
h\left(  s\right)  =h\left(  x\right)  +\left(  s-x\right)  h^{\prime}\left(
x\right)  +\frac{\left(  s-x\right)  ^{2}}{2!}h^{\prime\prime}\left(
\varsigma\right)
\]
where $\varsigma$ between $x$ and $s.\ $Then, by applying $\mathcal{S}_{n}$ to
this equality and using the linearity of the operator, we get%
\[
\mathcal{S}_{n}\left(  h;x\right)  -h\left(  x\right)  =h^{\prime}\left(
x\right)  \mathcal{S}_{n}\left(  s-x;x\right)  +\frac{h^{\prime\prime}\left(
\varsigma\right)  }{2}\Lambda_{2}.
\]
Using Lemma \ref{Lemma 4} and
\[
\mathcal{S}_{n}\left(  \left\vert s-x\right\vert ;x\right)  \leq\left(
\mathcal{S}_{n}\left(  \left(  s-x\right)  ^{2};x\right)  \right)  ^{\frac
{1}{2}}=\Lambda_{2}^{\frac{1}{2}},
\]
the following inequality is satisfied%
\[%
\begin{array}
[c]{l}%
\left\vert \mathcal{S}_{n}\left(  h;x\right)  -h\left(  x\right)  \right\vert
\leq\Lambda_{2}^{\frac{1}{2}}\left\Vert h^{\prime}\right\Vert _{C_{B}%
[0,\infty)}+\frac{\Lambda_{2}}{2}\left\Vert h^{\prime\prime}\right\Vert
_{C_{B}[0,\infty)}\\
\\
\leq\frac{\Lambda_{2}^{\frac{1}{2}}}{2}\left(  2+\Lambda_{2}^{\frac{1}{2}%
}\right)  \left\Vert h\right\Vert _{C_{B}^{2}[0,\infty)}.
\end{array}
\]

\end{proof}

For any $f\in C_{B}[0,\infty)$ and $\delta>0,$ Peetre's $K$-functional is
given by
\[
K_{2}\left(  f;\delta\right)  :=\inf\left\{  \left\Vert f-g\right\Vert
_{C_{B}[0,\infty)}+\delta\left\Vert g^{\prime\prime}\right\Vert _{C_{B}%
[0,\infty)}:~g\in C_{B}^{2}[0,\infty)\right\}  ,
\]
where $C_{B}^{2}[0,\infty)=\{g\in C_{B}[0,\infty):g^{\prime},g^{\prime\prime
}\in C_{B}[0,\infty)\}$ and the second modulus of continuity $\omega
_{2}\left(  f;\delta\right)  $ is defined as%
\begin{equation}
\omega_{2}\left(  f;\sqrt{\delta}\right)  :=\sup_{0<s\leq\sqrt{\delta}}%
\sup_{x\in\lbrack0,\infty)}\left\vert f\left(  x+2s\right)  -2f\left(
x+s\right)  +f\left(  x\right)  \right\vert . \label{b11}%
\end{equation}
Also, the inequality holds%
\begin{equation}
K_{2}(f;\delta)\leq c\omega_{2}(f;\sqrt{\delta})\ \ ,\ c>0, \label{*}%
\end{equation}
between Peetre's $K$-functional and second modulus of continuity $\omega_{2}$
(see \cite{DeVore-Lorentz}).

\begin{lemma}
For $g\in C_{B}^{2}[0,\infty)$ and $x\geq0,~\alpha,\lambda\geq0,$ we get%
\[
\left\vert \widetilde{\mathcal{S}}_{n}\left(  g;x\right)  -g\left(  x\right)
\right\vert \leq \Upsilon(n,x)\left\Vert g^{^{\prime\prime}}\right\Vert
_{C_{B}[0,\infty)}%
\]
where%
\[
\widetilde{\mathcal{S}}_{n}\left(  g;x\right)  =\mathcal{S}_{n}\left(
g;x\right)  +g(x)-g\left(  x+\tfrac{2\alpha x^{2}+\lambda+1}{n}\right)
\]
and%
\[
\Upsilon(n,x)=\Lambda_{1}^{2}+\Lambda_{2}.
\]

\end{lemma}

\begin{proof}
Because of the linearity property of the operator $\widetilde{\mathcal{S}}%
_{n}\left(  g;x\right)  =\mathcal{S}_{n}\left(  g;x\right)  +g(x)-g\left(
x+\tfrac{2\alpha x^{2}+\lambda+1}{n}\right)  ,$ for $t\in\lbrack0,\infty)$ we
get $\widetilde{\mathcal{S}}_{n}\left(  t-x;x\right)  =0.$ For $g\in C_{B}%
^{2}[0,\infty)$, the Taylor's expression is%
\[
g(t)=g(x)+(t-x)g^{^{\prime}}(x)+%
{\textstyle\int_{x}^{t}}
(t-v)g^{^{\prime\prime}}(v)dv,\ t\in\lbrack0,\infty).
\]
If we apply $\widetilde{\mathcal{S}}_{n}$ to the last equality and then use
$\widetilde{\mathcal{S}}_{n}\left(  t-x;x\right)  =0,$ we have%
\begin{align*}
\widetilde{\mathcal{S}}_{n}\left(  g;x\right)  -g(x)  &  =\widetilde
{\mathcal{S}}_{n}\left(
{\textstyle\int_{x}^{t}}
(t-v)g^{^{\prime\prime}}(v)dv;x\right) \\
&  =\mathcal{S}_{n}\left(
{\textstyle\int_{x}^{t}}
(t-v)g^{^{\prime\prime}}(v)dv;x\right)  -%
{\textstyle\int_{x}^{x+\tfrac{2\alpha x^{2}+\lambda+1}{n}}}
(x+\tfrac{2\alpha x^{2}+\lambda+1}{n}-v)g^{^{\prime\prime}}(v)dv,
\end{align*}
from which, it follows%
\begin{align}
\left\vert \widetilde{\mathcal{S}}_{n}\left(  g;x\right)  -g(x)\right\vert  &
\leq\left\vert \mathcal{S}_{n}\left(
{\textstyle\int_{x}^{t}}
(t-v)g^{^{\prime\prime}}(v)dv;x\right)  \right\vert \nonumber\\
&  +\left\vert
{\textstyle\int_{x}^{x+\tfrac{2\alpha x^{2}+\lambda+1}{n}}}
(x+\tfrac{2\alpha x^{2}+\lambda+1}{n}-v)g^{^{\prime\prime}}(v)dv\right\vert .
\label{AA}%
\end{align}
Since
\begin{equation}
\left\vert
{\textstyle\int_{x}^{t}}
(t-v)g^{^{\prime\prime}}(v)dv\right\vert \leq\left(  t-x\right)
^{2}\left\Vert g^{^{\prime\prime}}\right\Vert _{C_{B}[0,\infty)}, \label{A2}%
\end{equation}
we have
\begin{equation}
\left\vert
{\textstyle\int_{x}^{x+\tfrac{2\alpha x^{2}+\lambda+1}{n}}}
(x+\tfrac{2\alpha x^{2}+\lambda+1}{n}-v)g^{^{\prime\prime}}(v)dv\right\vert
\leq\left(  \tfrac{2\alpha x^{2}+\lambda+1}{n}\right)  ^{2}\left\Vert
g^{^{\prime\prime}}\right\Vert _{C_{B}[0,\infty)}. \label{A3}%
\end{equation}
From (\ref{AA}), (\ref{A2}) and (\ref{A3}),
\begin{align*}
\left\vert \widetilde{\mathcal{S}}_{n}\left(  g;x\right)  -g(x)\right\vert  &
\leq\left\{  \mathcal{S}_{n}\left(  (t-x)^{2};x\right)  +\left(
\tfrac{2\alpha x^{2}+\lambda+1}{n}\right)  ^{2}\right\}  \left\Vert
g^{^{\prime\prime}}\right\Vert _{C_{B}[0,\infty)}\\
&  =\Upsilon(n,x)\left\Vert g^{^{\prime\prime}}\right\Vert _{C_{B}[0,\infty)},
\end{align*}
where $\Upsilon(n,x)=\Lambda_{1}^{2}+\Lambda_{2}.$
\end{proof}

\begin{theorem}
Let $f\in C_{B}[0,\infty)$ and $c>0.$ The following inequality holds%
\[
\left\vert \mathcal{S}_{n}\left(  f;x\right)  -f\left(  x\right)  \right\vert
\leq c~\omega_{2}\left(  f;\frac{1}{2}\sqrt{\Upsilon(n,x)}\right)
+\omega\left(  f;\Lambda_{1}\right)  .
\]

\end{theorem}

\begin{proof}
For $f\in C_{B}[0,\infty)$ and $g\in C_{B}^{2}[0,\infty),$ we get%
\begin{align*}
\widetilde{\mathcal{S}}_{n}\left(  f-g;x\right)   &  =\mathcal{S}_{n}\left(
f-g;x\right)  +(f-g)(x)-(f-g)\left(  x+\tfrac{2\alpha x^{2}+\lambda+1}%
{n}\right)  \\
&  =\mathcal{S}_{n}\left(  f;x\right)  +f(x)-f\left(  x+\tfrac{2\alpha
x^{2}+\lambda+1}{n}\right)  -\widetilde{\mathcal{S}}_{n}\left(  g;x\right)  .
\end{align*}
On the other hand, we give%
\begin{align*}
\mathcal{S}_{n}\left(  f;x\right)  -f(x) &  =\widetilde{\mathcal{S}}%
_{n}\left(  f-g;x\right)  +\widetilde{\mathcal{S}}_{n}\left(  g;x\right)
-g(x)+g(x)+f\left(  x+\tfrac{2\alpha x^{2}+\lambda+1}{n}\right)  -f(x)-f(x)\\
&  =\widetilde{\mathcal{S}}_{n}\left(  f-g;x\right)  -\left(
f(x)-g(x)\right)  +\widetilde{\mathcal{S}}_{n}\left(  g;x\right)
-g(x)+f\left(  x+\tfrac{2\alpha x^{2}+\lambda+1}{n}\right)  -f(x).
\end{align*}
Using Lemma 6, thus we have%
\begin{align*}
\left\vert \mathcal{S}_{n}\left(  f;x\right)  -f(x)\right\vert  &
\leq\left\vert \widetilde{\mathcal{S}}_{n}\left(  f-g;x\right)  \right\vert
+\left\vert f(x)-g(x)\right\vert +\left\vert \widetilde{\mathcal{S}}%
_{n}\left(  g;x\right)  -g(x)\right\vert +\left\vert f\left(  x+\tfrac{2\alpha
x^{2}+\lambda+1}{n}\right)  -f(x)\right\vert \\
&  \leq4\left\Vert f-g\right\Vert _{C_{B}[0,\infty)}+\Upsilon(n,x)\left\Vert
g^{^{^{\prime\prime}}}\right\Vert _{C_{B}[0,\infty)}+\omega\left(
f;\Lambda_{1}\right)  .
\end{align*}
From the inequality (\ref{*}) between Peetre's $K$-functional and the second
modulus of continuity $\omega_{2}$, we have the desired result.
\end{proof}

\section{Data availability}

Data sharing not applicable to this article as no datasets were generated or
analysed during the current study.


\begin{thebibliography}{99}                                                                                               %


\bibitem {Ata-Buyuk}Atakut, \c{C}. and B\"{u}y\"{u}kyazici, \.{I}., Stancu
type generalization of the Favard Sz\'{a}sz operators, Appl. Math. Lett., 23
(12) \textbf{(2010)}, 1479-1482.

\bibitem {Atakut}Atakut, \c{C}. and \.{I}spir, N., Approximation by modified
Sz\'{a}sz--Mirakjan operators on weighted spaces, Proc. Indian Acad. Sci.
Math. 112 \textbf{(2002)}, 571--578.

\bibitem {Ciupa}Ciupa, A., A class of integral Favard--Sz\'{a}sz type
operators. Stud. Univ. Babes-Bolyai Math. 40 (1) \textbf{(1995)}, 39--47.

\bibitem {Gupta}Gupta, V., Vasishtha, V. and Gupta, M.K., Rate of convergence
of the Sz\'{a}sz--Kantorovich--Bezier operators for bounded variation
functions, Publ. Inst. Math. (Beograd) (N.S.), 72 \textbf{(2006)}, 137--143.

\bibitem {Szasz}Sz\'{a}sz, O., Generalization of S. Bernstein's polynomials to
the infinite interval, J. Res. Nat. Bur. Stand. 45 \textbf{(1950)}, 239--245.

\bibitem {Stancu}Stancu, D.D., Approximation of function by a new class of
polynomial operators, Rev. Rourn. Math. Pures et Appl., 13 (8) \textbf{(1968)}%
, 1173-1194.

\bibitem {ADT}Alt\i n, A., Do\u{g}ru, O. and Ta\c{s}delen, F., The
generalization of Meyer-K\"{o}nig and Zeller operators by generating
functions, J. Math. Anal. Appl., 312 (1) \textbf{(2005)}, 181-194.

\bibitem {DOT}Do\u{g}ru, O., \"{O}zarslan, M.A. and Ta\c{s}delen, F., On
positive operators involving a certain class of generating functions, Studia
Sci. Math. Hungar., 41 (4) \textbf{(2004)}, 415-429.

\bibitem {K}Krech, G., A note on some positive linear operators associated
with the Hermite polynomials, Carpathian J. Math., 32 (1) \textbf{(2016)}, 71--77.

\bibitem {OIT}Olgun, A., \.{I}nce, H. G. and Ta\c{s}delen, F.., Kantrovich
type generalization of Meyer-K\"{o}nig and Zeller operators via generating
functions, An. \c{S}tiin\c{t}. Univ. "Ovidius\textquotedblright\ Constan\c{t}a
Ser. Mat., 21 (3) \textbf{(2013)}, 209--221.

\bibitem {Sucu et al.}Sucu, S., \.{I}\c{c}\"{o}z, G. and Varma, S., On some
extensions of Sz\'{a}sz operators including Boas-Buck type polynomials, Abstr.
Appl. Anal., Vol. \textbf{2012}, Article ID 680340, 15 pages.

\bibitem {TAA}Ta\c{s}delen, F., Akta\c{s}, R., Alt\i n, A., A Kantrovich type
of Sz\'{a}sz operators including Brenke-type polynomials, Abstract and Applied
Analysis, Vol. 2012 \textbf{(2012)}, 13 pages.

\bibitem {Varma et al.}Varma, S., Sucu, S., \.{I}\c{c}\"{o}z, G.,
Generalization of Sz\'{a}sz operators involving Brenke type polynomials,
Comput. Math. Appl., 64 (2)\textbf{\ (2012)}, 121-127.

\bibitem {VT}Varma, S., Tasdelen F., Sz\'{a}sz type operators involving
Charlier polynomials, Mathematical and Computer Modeling, 56 \textbf{(2012)}, 118-122.

\bibitem {Sucu}Sucu, S., Dunkl analogue of Sz\'{a}sz operators, Appl. Math.
Comput. 244 \textbf{(2014)}, 42-48.

\bibitem {Rosenblum}Rosenblum, M., Generalized Hermite polynomials and the
Bose-like oscillator calculus, Oper. Theory: Adv. Appl. 73 \textbf{(1994)}, 369-396.

\bibitem {IC}\.{I}\c{c}\"{o}z, G., \c{C}ekim, B., Dunkl generalization of
Sz\'{a}sz operators via q-calculus, Journal of Inequalities and Applications,
2015:284 \textbf{(2015}), 11 pages.

\bibitem {icoz}\.{I}\c{c}\"{o}z, G., \c{C}ekim, B., Stancu-type generalization
of Dunkl analogue of Sz\'{a}sz--Kantorovich operators, Math. Meth. Appl. Sci,
39 \textbf{(2016)}, 1803--1810.

\bibitem {wafi}Wafi, A. and Rao, N., Sz\'{a}sz--gamma operators based on Dunkl
analogue, Iran J Sci Technol Trans Sci., 43(2019),213--223.

\bibitem {BRF}Akta\c{s}, R., \c{C}ekim, B. and Ta\c{s}delen, F., A Dunkl
Analogue of Operators Including Two-Variable Hermite polynomials, Bull.
Malays. Math. Sci. Soc., 42 \textbf{(2019)}, 2795--2805.

\bibitem {BC}Ben Cheikh, Y., Gaied, M. Dunkl--Appell d-orthogonal polynomials.
Integral Transforms and Special Functions, 18(8) \textbf{(2007)}, 581-597.

\bibitem {Appell}Appell, P. and Kampe de Feriet, J., Hypergeometriques et
Hyperspheriques: Polynomes d'Hermite, Gauthier-Villars, Paris, 1926.

\bibitem {Sucu2}Sucu, S., Approximation by sequence of operators including
Dunkl--Appell polynomials, Bull. Malays. Math. Sci. Soc., 43 (3)
\textbf{(2020)}, 2455-2464.

\bibitem {TSA}Ta\c{s}delen, F., S\"{o}ylemez, D., Akta\c{s}, R., Dunkl-Gamma
type operators including Appell polynomials, Complex Analysis and Operator
Theory, 13 \textbf{(2019)}, 3359--3371.

\bibitem {Milo}Milovanovi\'{c}, G. V., Mursaleen, M., \& Nasiruzzaman, M.
(2018). Modified Stancu type Dunkl generalization of Sz\'{a}sz--Kantorovich
operators. Revista de la Real Academia de Ciencias Exactas, F\'{\i}sicas y
Naturales. Serie A. Matem\'{a}ticas, 112(1), 135-151.

\bibitem {Korovkin}Korovkin, P.P., On convergence of linear positive operators
in the space of continuous functions (Russian). Doklady Akad. Nauk. SSSR (NS)
90 \textbf{(1953)}, 961--964

\bibitem {DeVore-Lorentz}DeVore, R.A. and Lorentz, G.G., Construtive
Approximation, Springer, Berlin, \textbf{1993}.
\end{thebibliography}
\end{document}